\documentclass[english]{amsart}
\usepackage[T1]{fontenc}
\usepackage[latin9]{inputenc}
\usepackage{babel}

\usepackage{verbatim}
\usepackage{amsthm, amstext, amssymb, amsfonts}
\usepackage[all]{xy}
\usepackage{graphicx}

\usepackage[bookmarksnumbered,plainpages,hypertex]{hyperref}
\usepackage{srcltx}

\numberwithin{equation}{section}
\numberwithin{figure}{section}
\theoremstyle{plain}
\newtheorem{thm}{Theorem}[section]
\theoremstyle{definition}
\newtheorem{defn}[thm]{Definition}
\theoremstyle{definition}
\newtheorem{example}[thm]{Example}
\theoremstyle{remark}
\newtheorem{claim}[thm]{Claim}
\theoremstyle{remark}
\newtheorem{rem}[thm]{Remark}
\theoremstyle{plain}
\newtheorem{lem}[thm]{Lemma}

\theoremstyle{plain}
\newtheorem{prop}[thm]{Proposition}
\theoremstyle{plain}
\newtheorem{cor}[thm]{Corollary}

\allowdisplaybreaks

\begin{document}

\title{Universal Enveloping Algebras of PBW Type}

\author{Alessandro Ardizzoni}

\curraddr{University of Ferrara, Department of Mathematics, Via Machiavelli
35, Ferrara, I-44121, Italy}

\email{alessandro.ardizzoni@unife.it}

\urladdr{http://www.unife.it/utenti/alessandro.ardizzoni}


\thanks{This paper was written while the author was member of GNSAGA with
partial financial support from MIUR within the National Research Project
PRIN 2007}

\begin{abstract}
We continue our investigation of the general notion of universal enveloping
algebra introduced in [A. Ardizzoni, \emph{A Milnor-Moore
Type Theorem for Primitively Generated Braided Bialgebras}, J. Algebra \textbf{327} (2011), no. 1, 337--365]. Namely we study when such an algebra is of PBW type, meaning that a suitable PBW type
theorem holds. We discuss the problem of finding a basis for a universal enveloping
algebra of PBW type: As an application we recover  the PBW basis both
of an ordinary universal enveloping algebra and of a restricted enveloping
algebra. We prove that a universal enveloping algebra is of PBW type if and only if it is cosymmetric. We characterize braided bialgebra liftings
of Nichols algebras as universal enveloping algebras of PBW type.
\end{abstract}


\maketitle
\tableofcontents

\section{Introduction}

Let $L$ be a Lie algebra which is assumed to have a totally ordered
basis $(X,\leq)$. A classical result from the theory of Lie algebras
asserts that the elements $x_{1}\cdots x_{n}$, where $n\geq1$, $x_{i}\in X$,
for all $1\leq i\leq n$, and $x_{1}\leq x_{2}\leq\cdots\leq x_{n}$,
along with $1$, form a basis of the universal enveloping algebra
$U(L)$ of $L$. This theorem is due to Poincaré, Birkhoff and Witt
and the basis is called the PBW basis of the universal enveloping
algebra (see e.g. \cite[Corollary C, page 92]{Humphreys}). This result
essentially relies on the existence of a bialgebra map $\omega:S\left(L\right)\rightarrow\mathfrak{G}\left(U(L)\right)$,
where $S(L)$ denotes the symmetric algebra on $L$ and $\mathfrak{G}\left(U(L)\right)$
is the graded braided bialgebra associated to the standard filtration of $U(L)$.
The fact that the map $\omega$ is bijective sometimes is called the
PBW theorem, see e.g. \cite[Corollary C, page 92]{Humphreys}.

Motivated by these observations, we intend to investigate a PBW type
Theorem for a general notion of universal enveloping algebra appeared
in \cite[Definition 5.2]{Ardizzoni-MMPrim}. To explain better this
notion we need to recall the definition of braided bialgebra.

Recall that a \emph{braided vector space} $(V,c)$ consists of a vector
space $V$ and a $K$-linear map $c:V\otimes V\rightarrow V\otimes V$,
called braiding, obeying the so-called quantum Yang-Baxter equation
$c_{1}c_{2}c_{1}=c_{2}c_{1}c_{2}$. Here $c_{1}=c\otimes V$ and $c_{2}=V\otimes c.$
A \emph{braided bialgebra} is then a braided vector space which is
both an algebra and a coalgebra with structures suitably compatible
with the braiding. Examples of braided bialgebras are all bialgebras
in those braided monoidal categories which are monoidal subcategories
of the category of vector spaces.

In \cite[Theorem 6.9]{Ardizzoni-MMPrim}, it is proved that every
primitively generated braided bialgebra is isomorphic, as a braided
bialgebra, to the generalized universal enveloping algebra $U(V,c,b)$
of its infinitesimal braided Lie algebra $(V,c,b)$ (here $b$ denotes
the bracket on the braided vector space $(V,c)$ consisting of primitive
elements in the given braided bialgebra). This result can be seen
as an extension of the celebrated Milnor-Moore Theorem \cite[Theorem 5.18]{Milnor-Moore}
for cocommutative connected bialgebras (once observed that such a
bialgebra is always primitively generated): in characteristic zero,
any cocommutative connected bialgebra is the enveloping algebra of
its space of primitive elements, regarded as a Lie algebra in a canonical
way. 

Now, in order to investigate a PBW type Theorem for $U(V,c,b)$ we
have to choose an appropriate substitute for the symmetric algebra.
Since in the classical case such an algebra is obtained as the universal
enveloping algebra of $L$ regarded as a Lie algebra through the trivial
bracket, the natural candidate is the Nichols algebra $\mathcal{B}\left(V,c\right)$
which is indeed of the form $U(V,c,b_{tr})$, where $b_{tr}$ denotes
the trivial bracket on $(V,c)$.\medskip{}

\textbf{The results in Detail.} The paper is organized as follows.

Section \ref{sec:prelim} contains preliminary facts and notations
that will be used in the paper.

Section \ref{sec:StandardCorad}, deals with the standard filtration
$(U_{(n)})_{n\in\mathbb{N}}$ on $U:=U(V,c,b)$. This filtration is
induced by the standard filtration on the braided tensor algebra $T(V,c)$
(the latter is just the tensor algebra $T(V)$ which is regarded as
a braided bialgebra through a comultiplication depending on the braiding
$c$). Having in mind the classical case, we consider the graded braided
bialgebra $\mathfrak{G}(U)$ associated to the standard filtration.
Denote by $\mathrm{gr}U$ the graded braided bialgebra associated
to the coradical filtration of $U$. In Proposition \ref{pro:CoradicalFiltration},
we produce a graded braided bialgebra homomorphism $\xi_{U}:\mathfrak{G}\left(U\right)\rightarrow\mathrm{gr}U$
and we characterize when this morphism is bijective. In Proposition
\ref{pro: induced filtration} and Proposition \ref{pro:omegaB},
we show there exist canonical graded braided bialgebra homomorphisms
$\vartheta_{U}:T\left(V,c\right)\rightarrow\mathfrak{G}\left(U\right)$
and $\chi_{U}:\mathcal{B}\left(V,c\right)\rightarrow\mathrm{gr}U$
such that the diagram

\begin{equation}
\xymatrix@R=10pt{T(V,c)\ar[rr]^{\Omega}\ar[d]_{\vartheta_{U}} &  & \mathcal{B}\left(V,c\right)\ar[d]^{\chi_{U}}\\
\mathfrak{G}(U)\ar[rr]^{\xi_{U}} &  & \mathrm{gr}U}
\label{diag:chi}\end{equation}
commutes, where $\Omega$ denotes the canonical projection. Moreover
$\vartheta_{U}$ is surjective and $\chi_{U}$ is injective. 

We say that $U$ is of \emph{PBW type} whenever $\vartheta_{U}$ quotients
to a braided bialgebra isomorphism $\omega_{U}:\mathcal{B}\left(V,c\right)\rightarrow\mathfrak{G}\left(U\right)$:
\begin{equation}
\xymatrix@R=10pt{T(V,c)\ar[rr]^{\Omega}\ar[rd]_{\vartheta_{U}} &  & \mathcal{B}\left(V,c\right)\ar@{.>}[ld]^{\omega_{U}}\\
 & \mathfrak{G}(U)}
\label{diag:omega}\end{equation}
Therefore $U$ is of PBW type means that $U$ fulfills a PBW type
Theorem. In Theorem \ref{thm:comb1PBW}, we prove that $U\left(V,c,b\right)$
is always of PBW type whenever $\left(V,c\right)$ has combinatorial
rank at most one in the sense of \cite[Definition 5.4]{Kharchenko-SkewPrim}. 

Section \ref{sec:PBWbasis} deals with the problem of determining
a basis for a universal enveloping algebra. In Proposition \ref{pro:PBWbasis},
mimicking classical ideas, we find a criterion that helps to obtain
a basis of $U=U(V,c,b)$ knowing a suitable basis of $\mathcal{B}\left(V,c\right)$
in case $U$ is of PBW type. This criterion is applied in Example
\ref{exa:PBWbasis-char0} and Example \ref{exa:PBWbasis-charp} to
recover the PBW basis both of an ordinary universal enveloping algebra
and of a restricted enveloping algebra. In Example \ref{exa:PBWbasis-Stumbo},
we compute a PBW basis for the universal enveloping algebra of a braided
Lie algebra whose bracket $c$ is not a symmetry i.e. $c^{2}\neq\mathrm{Id}$.

Section \ref{sec:Cosymm}, is devoted to the investigation of universal
enveloping algebras which are cosymmetric in the sense of \cite[Definition 3.1]{Kharchenko- connected}.
In Theorem \ref{teo:PBW}, using the results in \cite{Kharchenko- connected},
we give several characterizations of the fact that $U$ is of PBW type:
in particular this is equivalent to require that $U$ is cosymmetric.
In Theorem \ref{thm:KhaCosym}, we provide a sufficient condition
to have that $U$ is cosymmetric. This result is used in Corollary
\ref{cor:KhaCosym} to get that a braided vector space has combinatorial
rank at most $n+1$ whenever the corresponding symmetric algebra of
rank $n$ is cosymmetric. As a result, in Example \ref{ex:Cosymm},
we exhibit a universal enveloping algebra $U$ which is not of PBW
type.

In Section \ref{sec:Lifting}, we investigate braided bialgebra liftings
of the Nichols algebra. Explicitly, given a braided vector space $\left(V,c\right)$,
we say that a braided bialgebra $B$ is a lifting of\textbf{ $\mathcal{B}\left(V,c\right)$}
if there is a graded braided bialgebra isomorphism $\mathcal{B}\left(V,c\right)\cong\mathrm{gr}B$.
Theorem \ref{thm:Lifting} characterizes braided bialgebra liftings
of Nichols algebras as universal enveloping algebras of PBW type.
This result is applied in Corollary \ref{cor:Lifting} to Nichols
algebras algebras of a braided vector space of combinatorial rank
at most one. \medskip{}

\textbf{On PBW type theorems.} Several attempts to extend the classical PBW results to more general contexts appeared in the literature. Some of them, such as \cite{Yamane-APBWtheo, Rosso-AnAnaloguePBW, Lusztig-CanonicBasis, Ringel-PBW quantumGrps}, are related the quantized enveloping algebras $U_q(\mathfrak{g})$ of Drinfeld and Jimbo (note that this enveloping algebra is pointed but not connected whence it can not be described as a universal enveloping algebra of our kind). We now list some PBW type results which are closer to our approach, see also the references therein.

\begin{itemize}
\item A PBW Theorem for connected braided Hopf algebras with involutive braidings was obtained in \cite[Theorem 7.1]{Kharchenko- connected}. Our result can be seen as an extension of this one to the non-symmetric case.

\item A PBW Theorem for quadratic algebras can be found in \cite{Berger-ThequantumPBWThm} and in \cite{Braverman-Gaitsgory}. See also \cite[Theorem 3.9]{Ardizzoni-Stumbo-Qlie}.

\item Deep results on the PBW basis are obtained in \cite[Theorem 2]{Kherchenko-Aquantum} and more generally in \cite[Theorem 34]{Ufer}) for braided vector spaces of diagonal type or left triangular respectively. See also the more recent paper \cite{Helbig-APBWcriterion}.
\end{itemize}
We would like to point out that our aim here is not to compute explicitly a PBW basis for the Nichols algebra associated to a braided vector space. Instead we will give a method to produce a basis for a universal enveloping algebra $U(V,c,b)$ of PBW type once known a basis for the Nichols algebra $\mathcal{B}\left(V,c\right)$, see Remark \ref{rem:PBWbasis}.

\section{Preliminaries\label{sec:prelim}}

Throughout this paper $K$ will denote a field. All vector spaces
will be defined over $K$ and the tensor product over $K$ will be
denoted by $\otimes$.\medskip{}
 \\
 In this section we recall the main notions that we will deal with
in the paper. 
\begin{defn}
Let $V$ be a vector space over a field $K$. A $K$-linear map $c=c_{V}:V\otimes V\rightarrow V\otimes V$
is called a \textbf{braiding} if it satisfies the quantum Yang-Baxter
equation $c_{1}c_{2}c_{1}=c_{2}c_{1}c_{2}$ on $V\otimes V\otimes V$,
where we set $c_{1}:=c\otimes V$ and $c_{2}:=V\otimes c.$ The pair
$\left(V,c\right)$ will be called a\textbf{ braided vector space}.
A morphism of braided vector spaces $(V,c_{V})$ and $(W,c_{W})$
is a $K$-linear map $f:V\rightarrow W$ such that $c_{W}(f\otimes f)=(f\otimes f)c_{V}.$ 
\end{defn}
A general method for producing braided vector spaces is to take an
arbitrary braided category $(\mathcal{M},\otimes,K,a,l,r,c),$ which
is a monoidal subcategory of the category $\mathrm{Vect}_{K}$ of
$K$-vector spaces (here $a,l,r$ denote the associativity, the left
and the right unit constraints respectively). Hence any object $V\in\mathcal{M}$
can be regarded as a braided vector space with respect to $c:=c_{V,V}$,
where $c_{X,Y}:X\otimes Y\rightarrow Y\otimes X$ denotes the braiding
in $\mathcal{M}$, for all $X,Y\in\mathcal{M}$.

Let $\mathcal{N}$ be either the category of comodules over a coquasitriangular
Hopf algebra or the category of Yetter-Drinfeld modules over a Hopf
algebra with bijective antipode. Then the forgetful functor $F:\mathcal{N}\rightarrow\mathrm{Vect}_{K}$
is a strict monoidal functor. Hence $\mathcal{M}=\mathrm{Im}F$ is
an example of a category as above. 
\begin{defn}
\cite{Ba} A quadruple $(A,m,u,c)$ is called a \textbf{braided algebra}
if $(A,m,u)$ is an associative unital algebra, $(A,c)$ is a braided
vector space, $m$ and $u$ commute with $c$, that is the following
conditions hold:\begin{gather*}
c(m\otimes A)=(A\otimes m)(c\otimes A)(A\otimes c),\qquad c(A\otimes m)=(m\otimes A)\left(A\otimes c\right)(c\otimes A),\\
c(u\otimes A)=A\otimes u,\qquad c(A\otimes u)=u\otimes A.\end{gather*}
A morphism of braided algebras is, by definition, a morphism of ordinary
algebras which, in addition, is a morphism of braided vector spaces.
Similarly the notions of braided coalgebra and of morphism of braided
coalgebras is introduced. 

\cite[Definition 5.1]{Ta} A sextuple $(B,m,u,\Delta,\varepsilon,c)$
is a called a \textbf{braided bialgebra} if $(B,m,u,c)$ is a braided
algebra, $(B,\Delta,\varepsilon,c)$ is a braided coalgebra and the
following relation hold:\[
\Delta m=(m\otimes m)(B\otimes c\otimes B)(\Delta\otimes\Delta).\]

\end{defn}
Examples of the notions above are algebras, coalgebras and bialgebras
in any braided category $\mathcal{M}$ which is a monoidal subcategory
of $\mathrm{Vect}_{K}$. The notion of braided bialgebra admits a
graded counterpart which is called graded braided bialgebra. For further
results on this topic the reader is refereed to \cite[1.8]{AMS-MM2}.
\begin{example}
\label{ex: tensor algebra} Let $\left(V,c\right)$ be a braided vector
space. Consider the tensor algebra $T=T(V)$ with multiplication $m_{T}$
and unit $u_{T}$. This is a graded braided algebra with $n$th graded
component $T^n(V)=V^{\otimes n}$. The braiding $c_{T}$ on $T$ is defined using the the braiding $c$ of
$V$ (the graded component $c_T^{n,m}$ of $c_T$ is represented in Figure \ref{Fig:cT}, where each crossing stands for a copy of $c$).

\begin{figure}[h]\includegraphics[width=4cm,height=4cm]{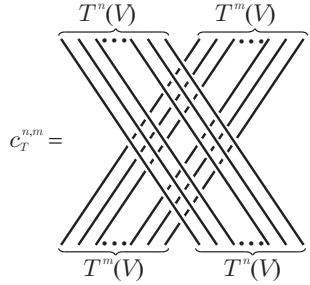}\caption{The braiding $c_T$}\label{Fig:cT}
\end{figure}

Now $T\otimes T$ becomes itself an algebra with multiplication $m_{T\otimes T}:=\left(m_{T}\otimes m_{T}\right)\left(T\otimes c_{T}\otimes T\right).$
This algebra is denoted by $T\otimes_{c}T.$ The universal property
of the tensor algebra yields two algebra homomorphisms $\Delta_{T}:T\rightarrow T\otimes_{c}T$
and $\varepsilon_{T}:T\rightarrow K$. It is straightforward to check
that $(T,m_{T},u_{T},\Delta_{T},\varepsilon_{T},c_{T})$ is a graded
braided bialgebra. Note that $\Delta_{T}$ really depends on $c$.
For example, one has $\Delta_{T}\left(z\right)=z\otimes1+1\otimes z+\left(c+\mathrm{Id}\right)\left(z\right)$,
for all $z\in V\otimes V$. \end{example}
\begin{defn}
The graded braided bialgebra described in Example \ref{ex: tensor algebra}
is called the \textbf{braided tensor algebra} and will be denoted
by $T(V,c)$. \end{defn}
\begin{claim}
\label{nn: connected} Recall that a coalgebra $C$ is called \textbf{connected}
if the coradical $C_{0}$ of $C$ (i.e the sum of all simple subcoalgebras
of $C$) is one-dimensional. In this case there is a unique group-like
element $1_{C}\in C$ such that $C_{0}=K1_{C}$. A morphism of connected
coalgebras is just a coalgebra homomorphisms (clearly it preserves
the group-like element).

By definition, a braided coalgebra $\left(C,c\right)$ is \textbf{connected}
if the underlying coalgebra is connected and, for any $x\in C$, $c(x\otimes1_{C})=1_{C}\otimes x$
and $c(1_{C}\otimes x)=x\otimes1_{C}$. \end{claim}
\begin{defn}
1) Let $B$ be a braided bialgebra with comultiplication $\Delta$
and unit $1_{B}$. Consider the space \[
P\left(B\right)=\left\{ b\in B\mid\Delta\left(b\right)=b\otimes1_{B}+1_{B}\otimes b\right\} \]
 of primitive elements in $B$. By \cite[Lemma 2.10]{Ardizzoni-Universal},
the braiding of $B$ induces a braiding $c_{P}$ of $P$ that will
be called the \textbf{infinitesimal braiding of} $B$. The braided
vector space $(P,c_{P})$ will be called the \textbf{infinitesimal
part of} $B$ (see \cite[Definition 2.11]{Ardizzoni-Universal}).

2) A braided bialgebra $B$ is called \textbf{primitively generated}
if it is generated as an algebra by $P\left(B\right)$. See \cite[page 239]{Milnor-Moore}.\end{defn}
\begin{rem}
\label{rem:ConnIsHopf}1) Let $B$ be a primitively generated braided
bialgebra. Then the underlying braided coalgebra is connected (see
\cite[Proposition 5.8]{Ardizzoni-Universal}). 

2) Let $B$ be a connected braided bialgebra. Since $B_{0}=K1_{B},$
then $\mathrm{Id}_{B}$ is convolution invertible in $\mathrm{Hom}(B_{0},B)$.
In view of the Takeuchi's result \cite[Lemma 5.2.10]{Montgomery},
we conclude that $\mathrm{Id}_{B}$ is convolution invertible in $\mathrm{Hom}(B,B)$.
Hence $B$ has an antipode i.e. it is a braided Hopf algebra (this
reasoning is similar to \cite[Remark 9.17]{Ardizzoni-SFFS} which
is due to Masuoka). In particular, by 1), any primitively generated
braided bialgebra is indeed a braided Hopf algebra. \end{rem}
\begin{defn}\label{def:Lie}
\cite[Definitions 5.2 and 5.4]{Ardizzoni-MMPrim} Let $\left(V,c\right)$
be a braided vector space. Suppose that, for each $n\in\mathbb{N}$,
there are a braided bialgebra $U^{[n]}$ and a map $i^{\left[n\right]}:V\rightarrow P\left(U^{\left[n\right]}\right)$
that fulfil the following requirements, where we set $P^{\left[n\right]}:=P\left(U^{\left[n\right]}\right)$
and $V^{\left[n\right]}:=i^{\left[n\right]}\left(V\right)$. 
\begin{itemize}
\item $U^{\left[0\right]}:=T\left(V,c\right)$ and $i^{\left[0\right]}:V\rightarrow P^{[0]}$
is the restriction of the canonical map $V\rightarrow U^{\left[0\right]}$. 
\item For each $n\in\mathbb{N}$, there exists a map $b^{\left[n\right]}:P^{[n]}\rightarrow V$
such that \begin{eqnarray}
c_{V^{\left[n\right]}}\left(i^{\left[n\right]}b^{\left[n\right]}\otimes V^{\left[n\right]}\right) & = & \left(V^{\left[n\right]}\otimes i^{\left[n\right]}b^{\left[n\right]}\right)c_{P^{\left[n\right]},V^{\left[n\right]}},\label{form:bracket1}\\
c_{V^{\left[n\right]}}\left(V^{\left[n\right]}\otimes i^{\left[n\right]}b^{\left[n\right]}\right) & = & \left(i^{\left[n\right]}b^{\left[n\right]}\otimes V^{\left[n\right]}\right)c_{V^{\left[n\right]},P^{\left[n\right]}},\label{form:bracket2}\end{eqnarray}
 where $c_{V^{\left[n\right]}}:V^{[n]}\otimes V^{[n]}\rightarrow V^{[n]}\otimes V^{[n]}$
and $c_{P^{\left[n\right]},V^{\left[n\right]}},:P^{[n]}\otimes V^{[n]}\rightarrow V^{[n]}\otimes P^{[n]}$
denote the restrictions of the braiding of $P^{[n]}$.
\item For each $n\in\mathbb{N}$,\[
U^{\left[n+1\right]}=\frac{U^{\left[n\right]}}{\left(\left[\mathrm{Id}-i^{\left[n\right]}b^{\left[n\right]}\right]\left[P^{\left[n\right]}\right]\right)},\]
 $\pi_n^{n+1}:U^{[n]}\rightarrow U^{[n+1]}$ is the natural projection of $U^{[n]}$ onto $U^{[n+1]}$ and $i^{\left[n+1\right]}:V\rightarrow P^{[n+1]}$ is the restriction
of the canonical map $V\rightarrow U^{\left[n+1\right]}$, so that $i^{[n+1]}=P(\pi_n^{n+1})\circ i^{[n]}$, where $P(\pi_n^{n+1}):P^{[n]}\rightarrow P^{[n+1]}$ is the map induced by $\pi_n^{n+1}$.
\end{itemize}
In this case we will say that $b:=\left(b^{\left[n\right]}\right)_{n\in\mathbb{N}}$\textbf{
}is a \textbf{bracket for a braided vector space $\left(V,c\right)$}.
This way we get a direct system of braided bialgebras \[
U^{\left[0\right]}\overset{\pi_{0}^{1}}{\rightarrow}U^{\left[1\right]}\overset{\pi_{1}^{2}}{\rightarrow}U^{\left[2\right]}\overset{\pi_{2}^{3}}{\rightarrow}\cdots.\]
 The direct limit of this direct system will be denoted by $\left(U\left(V,c,b\right),\pi_{U}\right):=\left(U^{\left[\infty\right]},\pi_{n}^{\infty}\right),$ where $\pi_{U}:T(V,c)\rightarrow U\left(V,c,b\right)$ is the canonical projection.
Now $U\left(V,c,b\right)$ becomes a primitively generated braided
bialgebra which will be called \textbf{the} \textbf{universal
enveloping algebra of} $\left(V,c,b\right).$ Denote by $i_{U}:V\rightarrow U\left(V,c,b\right)$
the (not necessarily injective) canonical map and set $P^{\left[\infty\right]}:=P\left(U^{\left[\infty\right]}\right).$
Note that $\mathrm{Im}\left(i_{U}\right)\subseteq P^{\left[\infty\right]}$
so that $i_{U}$ induces a morphism of braided vector spaces $i^{\left[\infty\right]}:V\rightarrow P^{\left[\infty\right]}$.

We say that $\left(V,c,b\right)$ is a \textbf{braided Lie algebra}
whenever $\left(V,c\right)$ is a braided vector space, $b$ is a
bracket on $\left(V,c\right)$ and $i_{U}:V\rightarrow U\left(V,c,b\right)$
is injective.

Let $B$ be a braided bialgebra. By \cite[Theorem 6.5]{Ardizzoni-MMPrim},
the infinitesimal part $\left(P,c_{P}\right)$ of $B$ is endowed
with a specific bracket $b_{P}$ such that $\left(P,c_{P},b_{P}\right)$
is a braided Lie algebra. $\left(P,c_{P},b_{P}\right)$ will be called
the \textbf{infinitesimal braided Lie algebra of} $B.$ \end{defn}
\begin{rem}
\label{rem:Milnor-Moore} 1) Let $\left(V,c,b\right)$ be a braided
Lie algebra and set $U:=U\left(V,c,b\right)$. In view of \cite[Corollary 5.6]{Ardizzoni-MMPrim},
the map $i_{U}:V\rightarrow U$ induces an isomorphism between $V$
and $P\left(U\right)$. 

2) By \cite[Theorem 6.9]{Ardizzoni-MMPrim}, every primitively generated
braided bialgebra is isomorphic as a braided bialgebra to the universal
enveloping algebra of its infinitesimal braided Lie algebra.\end{rem}

\begin{rem}
Let $(V,c,b)$ be a braided Lie algebra.
When $c$ is a symmetry, i.e. $c^{2}=\mathrm{Id}_{V\otimes
V}, $ and the characteristic of $K$ is zero, our universal enveloping algebra $U(V,c,b)$ reduces
to the one introduced in \cite{Gu- Gen Trans Lie} (cf. \cite[Remark 6.4]{Ardizzoni-Universal} using Remark \ref{rem:Universal} below). Other notions of Lie
algebra and universal enveloping algebra, extending the ones in \cite{Gu- Gen Trans Lie} to the non-symmetric case, appeared in the literature. Let us mention some of them without any pretension of exhaustiveness.

\begin{itemize}
\item Lie algebras for braided vector spaces $\left( V,c\right) $ where $c$
is a braiding of Hecke type \cite[Definition 7.1]{Wa}. Compare with \cite[Section 6]{Ardizzoni-Universal} using  Remark \ref{rem:Universal} to see how our notion of universal enveloping algebra behaves in this setting.


\item Lie algebras for braided vector spaces $\left( V,c\right) $ where $c$
is constructed by means of braidings of Hecke type
\cite[Definition 1]{Gurevich: Hecke sym}. See \cite[Remark 3.6]{Ardizzoni-Stumbo-Qlie} using  Remark \ref{rem:Universal}.

\item Braided Lie algebras for objects $L$ in a braided or quasi-tensor category equipped with a coproduct, a counit and a bracket $[,]\colon L\otimes L\rightarrow L$ satisfying some axioms, \cite[Definition 4.1]{Majid-QuantumBrLiealg}. Here the universal enveloping algebra is defined as the quotient of the tensor algebra over $L$ modulo some quadratic relations.

\item Quantum Lie algebras for objects $\mathfrak{g}$ in a monoidal category equipped with a braiding $\sigma:\mathfrak{g}\otimes \mathfrak{g}\rightarrow \mathfrak{g}\otimes \mathfrak{g}$ and a bracket $[,]\colon \mathfrak{g}\otimes \mathfrak{g}\rightarrow \mathfrak{g}$ satisfying some axioms, see \cite[Definition 2.1]{Gomez-Majid-BraidedLie}, which naturally arises in the context of covariant differential calculus over quantum groups \cite[Theorems 5.3 and 5.4]{Woronowicz-DifCalcCompactMatr}.

\item Lie algebras for objects in the braided monoidal category of
Yetter-Drinfeld modules over a Hopf algebra with bijective antipode \cite[%
Definition 4.1]{Pareigis-On Lie}. Compare with \cite[Section 8]{Ardizzoni-Universal} using  Remark \ref{rem:Universal}.

\item Lie algebras defined by considering quantum operations (see \cite[%
Definition 2.2]{Kharchenko-AnAlgSkew}) as primitive polynomials in the
tensor algebra. When the underline braided vector space is an object in the category of Yetter-Drinfeld modules over some group algebra, to any Lie algebra of this kind, Kharchenko associates a universal
enveloping algebra which is not connected (see \cite{Kharchenko-SkewPrim}). Note that the universal
enveloping algebra we deal with in the present paper is connected as it is meant to describe the structure of primitively generated (whence connected)
braided bialgebras over $K$, see 2) in Remark \ref{rem:Milnor-Moore}.
\end{itemize}
\end{rem}

\begin{defn}
\label{def:SymAlg}In view of \cite[Example 7.1]{Ardizzoni-MMPrim},
any braided vector space $\left(V,c\right)$ can be endowed with the
so-called \textbf{trivial bracket} $b_{tr}$, which makes of $\left(V,c,b_{tr}\right)$
a braided Lie algebra. Set $U:=U\left(V,c,b_{tr}\right)$. Then $U^{[n]}$
is called the \textbf{symmetric algebra of rank $n$ of }$\left(V,c\right)$
and is denoted by $S^{[n]}$. It is the braided bialgebra $S^{\left[n\right]}\left(B\right)$
introduced in \cite[Definition 3.10]{Ardizzoni-Sdeg} in the case
$B=T(V,c)$. Explicitly, $S^{[0]}=T(V,c)$ and, for all $n>0$, $S^{[n]}$
is the graded braided bialgebra obtained dividing out $S^{[n-1]}$
by the two-sided ideal generated by the homogeneous primitive elements
in $S^{[n-1]}$ of degree at least two. Moreover $U^{[\infty]}=U$ is denoted by $S^{[\infty]}$ and
identifies with the \textbf{Nichols algebra} $\mathcal{B}\left(V,c\right)$
(see \cite[5.3]{Ardizzoni-Sdeg} for a different definition). We will
denote by $\Omega:T(V,c)\rightarrow\mathcal{B}\left(V,c\right)$ the
canonical projection. 

If there exists a least $n\in\mathbb{N}$ such that $S^{[n]}=S^{[\infty]}$,
then we will say that $(V,c)$ has \textbf{combinatorial rank $n$}
(cf. \cite[Definition 5.4]{Kharchenko-SkewPrim}, see also \cite[Section 5]{Ardizzoni-Sdeg}). 
\end{defn}

We include here a technical result about $P^{[n]}$ that will be used to prove Theorem \ref{thm:KhaCosym}.
\begin{lem}\label{lem:Psplits} 
 Let $\left(V,c,b\right)$ be a braided Lie algebra. Let $P^{\left[n\right]}$ and $V^{\left[n\right]}$ be the spaces introduces in Definition \ref{def:Lie}. Set $W^{\left[n\right]}:=\mathrm{Ker}\left(\pi_{n}^{n+1}\right)\cap P^{\left[n\right]}$. Then $W^{\left[n\right]}=\mathrm{Ker}\left(b^{\left[n\right]}\right)=\mathrm{Im}\left(\mathrm{Id}_{P^{\left[n\right]}}-i^{\left[n\right]}b^{\left[n\right]}\right)$
and $P^{\left[n\right]}=W^{\left[n\right]}\oplus V^{\left[n\right]}$,
as a direct sum of braided subspaces.
\end{lem}

\begin{proof}
As in Definition \ref{def:Lie}, denote by $P\left(\pi_{n}^{n+1}\right):P^{[n]}\rightarrow P^{[n+1]}$ the natural map induced by $\pi_{n}^{n+1}$. Since $\left(V,c,b\right)$ is a braided Lie algebra, by \cite[Proposition 5.7]{Ardizzoni-MMPrim},
we have that $b^{\left[n+1\right]}P\left(\pi_{n}^{n+1}\right)=b^{\left[n\right]}$
and $b^{\left[n\right]}i^{\left[n\right]}=\mathrm{Id}_{P^{\left[n\right]}}$.
By the first equality, for $w\in W^{\left[n\right]}$, we have \[
0=b^{\left[n+1\right]}\pi_{n}^{n+1}\left(w\right)\overset{w\in P^{\left[n\right]}}{=}b^{\left[n+1\right]}P\left(\pi_{n}^{n+1}\right)\left(w\right)=b^{\left[n\right]}\left(w\right)\]
so that $W^{\left[n\right]}\subseteq\mathrm{Ker}\left(b^{\left[n\right]}\right)$.
On the other hand, if $z\in\mathrm{Ker}\left(b^{\left[n\right]}\right)$,
then $z\in P^{\left[n\right]}$ and \[
\pi_{n}^{n+1}\left(z\right)\overset{z\in P^{\left[n\right]}}{=}\pi_{n}^{n+1}i^{\left[n\right]}b^{\left[n\right]}\left(z\right)=0\]
so that $\mathrm{Ker}\left(b^{\left[n\right]}\right)\subseteq W^{\left[n\right]}$
whence $W^{\left[n\right]}=\mathrm{Ker}\left(b^{\left[n\right]}\right)$.
Now, from $b^{\left[n\right]}i^{\left[n\right]}=\mathrm{Id}_{P^{\left[n\right]}}$,
one gets $W^{\left[n\right]}=\mathrm{Ker}\left(b^{\left[n\right]}\right)=\mathrm{Im}\left(\mathrm{Id}_{P^{\left[n\right]}}-i^{\left[n\right]}b^{\left[n\right]}\right)$
and $P^{\left[n\right]}=W^{\left[n\right]}\oplus V^{\left[n\right]}$,
as a direct sum of braided subspaces.
\end{proof}

Finally, for the reader's convenience, we quote a technical result we will invoke three times in the paper.

\begin{lem}\label{lem:Heyneman-Radford}
\cite[Lemma 5.3.3]{Montgomery} If $C$ is connected and $f:C\rightarrow D$ is a coalgebra map such that $f_{\mid P(C)}$ is injective, then $f$ is injective.
\end{lem}

\section{Standard vs coradical filtration\label{sec:StandardCorad}}

In this section we introduce the standard filtration of a universal
enveloping algebra and we study it in connection with the coradical
filtration of the underlying coalgebra.
\begin{defn}
Recall that a filtration on a vector space $M$ is an increasing sequence
$M_{\left(0\right)}\subseteq M_{\left(1\right)}\subseteq\cdots\subseteq M_{\left(n\right)}\subseteq\cdots$
of subspaces of $M$. By convention we write $M_{(-1)}:=0$. A filtration
$\left(M_{\left(n\right)}\right)_{n\in\mathbb{N}}$ on a vector space
$M$ gives rise to a graded module \[
\mathfrak{G}\left(M\right):=\oplus_{n\in\mathbb{N}}\mathfrak{G}^{n}\left(M\right)\qquad\textrm{where}\qquad\mathfrak{G}^{n}\left(M\right):=\frac{M_{\left(n\right)}}{M_{\left(n-1\right)}}.\]
Let $M$ and $N$ be filtered vector spaces with filtrations $\left(M_{\left(n\right)}\right)_{n\in\mathbb{N}}$
and $\left(N_{\left(n\right)}\right)_{n\in\mathbb{N}}$ respectively.
A filtered homomorphism is a $K$-linear map $f:M\rightarrow N$ such
that $f\left(M_{\left(n\right)}\right)\subseteq N_{\left(n\right)}$
for all $n\in\mathbb{N}$. Such a morphism induces in a natural way
a graded homomorphism $\mathfrak{G}\left(f\right):\mathfrak{G}\left(M\right)\rightarrow\mathfrak{G}\left(N\right).$
The $n$th graded component of $\mathfrak{G}\left(f\right)$ will
be denoted by $\mathfrak{G}^{n}\left(f\right)$.

A braided bialgebra $\left(B,c_{B}\right)$ is called \textbf{filtered}
if the underlying vector space has a filtration $\left(B_{\left(n\right)}\right)_{n\in\mathbb{N}}$
with $B=\cup B_{\left(n\right)}$ such that \[
\Delta_{B}\left(B_{\left(n\right)}\right)\subseteq\sum\limits _{i=0}^{n}B_{\left(i\right)}\otimes B_{\left(n-i\right)},\qquad B_{\left(i\right)}\cdot_{B}B_{\left(j\right)}\subseteq B_{\left(i+j\right)}\qquad\text{and}\qquad c_{B}\left(B_{\left(i\right)}\otimes B_{\left(j\right)}\right)\subseteq B_{\left(j\right)}\otimes B_{\left(i\right)}\]
for all $i,j\in\mathbb{N}$. A filtered braided bialgebra homomorphism
$f:\left(B,c_{B}\right)\rightarrow\left(B^{\prime},c_{B^{\prime}}\right)$
is a filtered homomorphism $f:B\rightarrow B^{\prime}$ which is also
a braided bialgebra homomorphism.\end{defn}
\begin{lem}
\label{lem: graded of filtered}Let $\left(B,c_{B}\right)$ be a filtered
braided bialgebra with filtration $\left(B_{\left(n\right)}\right)_{n\in\mathbb{N}}$.
Assume $B_{\left(0\right)}=K.$ Then the space $\mathfrak{G}\left(B\right)$
is a graded braided bialgebra with structures induced by those of
$B$. Furthermore, any filtered braided bialgebra homomorphism $f:\left(B,c_{B}\right)\rightarrow\left(B^{\prime},c_{B^{\prime}}\right)$
induces a graded braided bialgebra homomorphism $\mathfrak{G}\left(f\right):\mathfrak{G}\left(B\right)\rightarrow\mathfrak{G}\left(B^{\prime}\right).$ \end{lem}
\begin{proof}
Set $G:=\mathfrak{G}\left(B\right),$ $G^{n}:=B_{\left(n\right)}/B_{\left(n-1\right)}$
and let $p_{n}:B_{\left(n\right)}\rightarrow G^{n}$ be the canonical
projection, for every $n\in\mathbb{N}$. Since $\left(B_{\left(n\right)}\right)_{n\in\mathbb{N}}$
is a coalgebra filtration on $B,$ then $G$ carries a graded coalgebra
structure $\left(G,\Delta_{G},\varepsilon_{G}\right)$ (see \cite[page 230]{Sw}).
Moreover the coradical of $B$ is contained in $B_{\left(0\right)}$
(see \cite[Proposition 11.1.1]{Sw}) whence it is $K.$ In particular
$1_{B}\in B_{\left(0\right)}$. Thus $\left(B_{\left(n\right)}\right)_{n\in\mathbb{N}}$
is also an algebra filtration on $B$ whence $G$ carries a graded
algebra structure too (see \cite[page 230]{Sw}). Since by definition
$c_{B}\left(B_{\left(i\right)}\otimes B_{\left(j\right)}\right)\subseteq B_{\left(j\right)}\otimes B_{\left(i\right)}$
for all $i,j\in\mathbb{N}$, then $c_{B}$ induces a braiding $c_{G}^{a,b}:G^{a}\otimes G^{b}\rightarrow G^{b}\otimes G^{a}$
for all $a,b\in\mathbb{N}.$ It is straightforward to prove that $\left(G,m_{G},u_{G},\Delta_{G},\varepsilon_{G},c_{G}\right)$
is indeed a graded braided bialgebra. Let $f:\left(B,c_{B}\right)\rightarrow\left(B^{\prime},c_{B^{\prime}}\right)$
be a filtered braided bialgebra homomorphism. By \cite[page 229]{Sw},
$f$ induces a graded coalgebra map $\mathfrak{G}\left(f\right):\mathfrak{G}\left(B,c_{B}\right)\rightarrow\mathfrak{G}\left(B^{\prime},c_{B^{\prime}}\right)$.
Furthermore, the same map is also a graded algebra map. It is easy
to check that $\mathfrak{G}\left(f\right)$ is also a morphism of
braided vector spaces whence a graded braided bialgebra homomorphism. 
\end{proof}
Let $\left(V,c\right)$ be a braided vector space and set $T:=T\left(V,c\right)$.
Recall that the \textbf{\emph{standard filtration}} on $T$ is defined
by setting $T_{\left(n\right)}:=\oplus_{i=0}^{n}V^{\otimes i}.$ 
\begin{lem}
\label{lem: induced filtration}Let $\left(V,c,b\right)$ be a braided
bialgebra and set $U:=U(V,c,b)$. Let $\left(T_{\left(n\right)}\right)_{n\in\mathbb{N}}$
be the standard filtration on $T\left(V,c\right)$. Set $U_{\left(n\right)}:=\pi_{U}\left(T_{\left(n\right)}\right)$,
for each $n\in\mathbb{N}$. Then $\left(U_{\left(n\right)}\right)_{n\in\mathbb{N}}$
is a filtration on $U$ that makes it a filtered connected braided
bialgebra. \end{lem}
\begin{proof}
It is straightforward.\end{proof}
\begin{defn}
Let $\left(V,c,b\right)$ be a braided Lie algebra and set $U:=U(V,c,b)$. 

The filtration $\left(U_{\left(n\right)}\right)_{n\in\mathbb{N}}$
of Lemma \ref{lem: induced filtration} will be called the \textbf{standard
filtration} on $U$ and $\mathfrak{G}\left(U\right)$ will denote
the graded braided bialgebra associated to this filtration.

We say that $U$ is \textbf{strictly generated by} $V$ whenever the
standard and the coradical filtration on $U$ coincide. This means
the $n$th term $U_n$ of the coradical filtration of $U$ is given by $U_{n}=U_{\left(n\right)}=\sum_{i=0}^{n}\pi_{U}\left(V\right)^{i}$. \end{defn}
\begin{rem}
\label{rem:strictisprimgen} Let $B$ be a braided bialgebra. If $B$
is connected, then the coradical filtration $\left(B_{n}\right)_{n\in\mathbb{N}}$
of $B$ makes $B$ itself into a filtered connected braided bialgebra. Thus,
by Lemma \ref{lem: graded of filtered}, the graded coalgebra $\mathrm{gr}B$
associated to the coradical filtration of $B$ is indeed a graded
braided bialgebra. Since $B$ is connected, by \cite[Proposition 11.1.1]{Sw} we have that $\mathrm{gr}B$ is connected too. Thus, by \cite[Lemma 11.2.3]{Sw}, $\mathrm{gr}B$ is strictly graded as a coalgebra. In particular $P(\mathrm{gr}B)=\mathrm{gr}^1B=B_1/B_0$.\end{rem}
\begin{prop}
\label{pro:CoradicalFiltration}Let $\left(V,c,b\right)$ be a braided
Lie algebra and set $U:=U(V,c,b)$. Then the identity on $U$ induces
a graded braided bialgebra homomorphism $\xi_{U}:\mathfrak{G}\left(U\right)\rightarrow\mathrm{gr}U.$
The following assertions are equivalent.
\begin{enumerate}
\item $\xi_{U}$ is bijective.
\item $\xi_{U}$ is surjective.
\item $\xi_{U}$ is injective.
\item $U_{\left(n\right)}\cap U_{n-1}=U_{\left(n-1\right)}$ for
all $n\in\mathbb{N}$.
\item $U_{n}=U_{\left(n\right)}+U_{n-1}$ for all $n\in\mathbb{N}$.
\item $U$ is strictly generated by $V$.
\item $P\left(\mathfrak{G}\left(U\right)\right)=U_{\left(1\right)}/U_{\left(0\right)}$.
\item $\pi_{U}\left(V\right)^{n}\cap U_{n-1}\subseteq U_{\left(n-1\right)}$
for all $n\in\mathbb{N}$.
\end{enumerate}
\end{prop}
\begin{proof}
Note that, for each $n\in\mathbb{N}$, one has $U_{\left(n\right)}\subseteq U_{n}$
so that $\mathrm{Id}_{U}$ is a filtered braided bialgebra automorphism
whence, by Lemma \ref{lem: graded of filtered}, it induces a graded
braided bialgebra homomorphism $\xi_{U}:\mathfrak{G}\left(U\right)\rightarrow\mathrm{gr}U.$ 

$\left(6\right)\Rightarrow\left(1\right)\Rightarrow\left(2\right),\left(1\right)\Rightarrow\left(3\right)$
These implications are trivial.

$\left(3\right)\Leftrightarrow\left(4\right)$ Let $\xi_{U}^{n}:U_{\left(n\right)}/U_{\left(n-1\right)}\rightarrow U_{n}/U_{n-1}$
be the $n$th graded component of $\xi_{U}$. Then $\mathrm{Ker}\left(\xi_{U}^{n}\right)=\left[U_{\left(n\right)}\cap U_{n-1}\right]/U_{\left(n-1\right)}$. 

$\left(4\right)\Rightarrow\left(6\right)$ It is enough to prove that
$U_{n}\subseteq U_{\left(n\right)}$. Let $z\in U_{n}$. Since $\pi_U\left(V\right)$
generates $U$ as a $K$-algebra, there is a least $t\in\mathbb{N}$
such that $z\in U_{\left(t\right)}$. If $t\geq n+1$, then $z\in U_{\left(t\right)}\cap U_{n}\subseteq U_{\left(t\right)}\cap U_{t-1}=U_{\left(t-1\right)}$
contradicting the minimality of $t$. Then $t\leq n$ whence $z\in U_{\left(t\right)}\subseteq U_{\left(n\right)}$. 

$\left(2\right)\Leftrightarrow\left(5\right)$ It follows from $\xi_{U}\left(U_{\left(n\right)}/U_{\left(n-1\right)}\right)=\left[U_{\left(n\right)}+U_{n-1}\right]/U_{n-1}$.

$\left(5\right)\Rightarrow\left(6\right)$ By induction on $n\in\mathbb{N}$,
we deduce that $U_{n}=U_{\left(n\right)}$.

$\left(1\right)\Rightarrow\left(7\right)$ Since $U$ is connected,
by Remark \ref{rem:strictisprimgen}, the primitive
part of $\mathrm{gr}U$ is $\mathrm{gr}^{1}U$. Since $\xi_{U}$ is
a graded braided bialgebra isomorphism we get that $P\left(\mathfrak{G}\left(U\right)\right)=\mathfrak{G}^{1}\left(U\right).$

$\left(7\right)\Rightarrow\left(3\right)$ In view of the hypothesis,
the restriction of $\xi_{U}$ to $P\left(\mathfrak{G}\left(U\right)\right)$
is the map $\xi_{U}^{1}:U_{\left(1\right)}/U_{\left(0\right)}\rightarrow U_{1}/U_{0}.$
The kernel of this map is $\left[U_{\left(1\right)}\cap U_{0}\right]/U_{\left(0\right)}$.
Since $U$ is connected, then $U_{0}=K=U_{\left(0\right)}$ so that
$U_{\left(1\right)}\cap U_{0}=U_{\left(1\right)}\cap U_{\left(0\right)}=U_{\left(0\right)}$.
Hence the restriction of $\xi_{U}$ to $P\left(\mathfrak{G}\left(U\right)\right)$
is injective. By Lemma \ref{lem:Heyneman-Radford}, $\xi_{U}$ is injective
too.

$\left(4\right)\Leftrightarrow\left(8\right)$ It is enough to prove
that $U_{\left(n\right)}\cap U_{n-1}=U_{\left(n-1\right)}+\pi_{U}\left(V\right)^{n}\cap U_{n-1}$.
Let $z\in U_{\left(n\right)}\cap U_{n-1}$. Then $z\in U_{\left(n\right)}=U_{\left(n-1\right)}+\pi_{U}\left(V\right)^{n}$.
Hence there are $x\in U_{\left(n-1\right)}$ and $y\in\pi_{U}\left(V\right)^{n}$
such that $z=x+y$. Then $y=z-x\in U_{\left(n\right)}\cap U_{n-1}+U_{\left(n-1\right)}\subseteq U_{n-1}$.
We have so proved that $z=x+y\in U_{\left(n-1\right)}+\pi_{U}\left(V\right)^{n}\cap U_{n-1}$
so that $U_{\left(n\right)}\cap U_{n-1}\subseteq U_{\left(n-1\right)}+\pi_{U}\left(V\right)^{n}\cap U_{n-1}$.
The other inclusion is trivial.\end{proof}
\begin{prop}
\label{pro: induced filtration}Let $\left(V,c,b\right)$ be a braided
bialgebra and set $U:=U(V,c,b)$. Let $\left(U_{\left(n\right)}\right)_{n\in\mathbb{N}}$
be the standard filtration on $U$. Then there exists a graded braided
bialgebra homomorphism $\vartheta_{U}:T\left(V,c\right)\rightarrow\mathfrak{G}\left(U\right)$
which is surjective and lifts the map $\vartheta_{U}^{1}:V\rightarrow U_{\left(1\right)}/U_{\left(0\right)}=\mathfrak{G}^{1}\left(U\right):v\mapsto\pi_{U}\left(v\right)+U_{\left(0\right)}$. \end{prop}
\begin{proof}
Set $T:=T\left(V,c\right)$ and let $q_{n}:U_{\left(n\right)}\rightarrow U_{\left(n\right)}/U_{\left(n-1\right)}\mathfrak{=G}^{n}\left(U\right)$
denote the canonical projection. Since $\mathfrak{G}\left(U\right)$
is a connected graded coalgebra, it is clear that $\mathrm{Im}\left(\vartheta_{U}^{1}\right)=\mathfrak{G}^{1}\left(U\right)\subseteq P\left(\mathfrak{G}\left(U\right)\right)$.
Moreover $\vartheta_{U}^{1}:V\rightarrow P\left(\mathfrak{G}\left(U\right)\right)$
is a morphism of braided vector spaces as, for every $u,v\in V$ we
have\begin{eqnarray*}
c_{\mathfrak{G}\left(U\right)}\left(\vartheta_{U}^{1}\otimes\vartheta_{U}^{1}\right)\left(u\otimes v\right) & = & c_{\mathfrak{G}\left(U\right)}\left((\pi_{U}\left(u\right)+U_{\left(0\right)})\otimes (\pi_{U}\left(v\right)+U_{\left(0\right)})\right)\\
 & = & \left(q_{1}\otimes q_{1}\right)c_{U}\left(\pi_{U}\left(u\right)\otimes\pi_{U}\left(v\right)\right)\\
 & = & \left(q_{1}\otimes q_{1}\right)\left(\pi_{U}\otimes\pi_{U}\right)c_{U}\left(u\otimes v\right)=\left(\vartheta_{U}^{1}\otimes\vartheta_{U}^{1}\right)c\left(u\otimes v\right).\end{eqnarray*}
By the universal property of the braided tensor algebra there exists
a graded braided bialgebra homomorphism $\vartheta_{U}:T\left(V,c\right)\rightarrow\mathfrak{G}\left(U\right)$
that restricted to $V$ yields $\vartheta_{U}^{1}$. Since $\mathfrak{G}\left(U\right)$
is generated as a $K$-algebra by $\mathfrak{G}^{1}\left(U\right)=\mathrm{Im}\left(\vartheta_{U}^{1}\right)$,
we infer that $\vartheta_{U}$ is surjective.\end{proof}
\begin{prop}
\label{pro:omegaB}Let $\left(V,c,b\right)$ be a braided Lie algebra
and set $U:=U(V,c,b)$. Let $\vartheta_{U}:T\left(V,c\right)\rightarrow\mathfrak{G}\left(U\right)$
be the map of Proposition \ref{pro: induced filtration}. Then $\xi_{U}\vartheta_{U}:T\left(V,c\right)\rightarrow\mathrm{gr}U$
is the unique graded braided bialgebra homomorphism lifting the map
$\chi_{U}^{1}:V\rightarrow U_{1}/U_{0}=:v\mapsto\pi_{U}\left(v\right)+U_{0}$.
Moreover $\xi_{U}\vartheta_{U}$ quotients to an injective braided
bialgebra homomorphism $\chi_{U}:\mathcal{B}\left(V,c\right)\rightarrow\mathrm{gr}U$
i.e. \eqref{diag:chi} commutes.\end{prop}
\begin{proof}
Let $\left(U_{\left(n\right)}\right)_{n\in\mathbb{N}}$ be the standard
filtration on $U$. By Proposition \ref{pro: induced filtration},
$\vartheta_{U}:T\left(V,c\right)\rightarrow\mathfrak{G}\left(U\right)$
is surjective and lifts the map $\vartheta_{U}^{1}$. Then $\xi_{U}\vartheta_{U}:T\left(V,c\right)\rightarrow\mathrm{gr}U$
is the unique graded braided bialgebra homomorphism from $T\left(V,c\right)$
to $\mathrm{gr}U$ which lifts the map $\chi_{U}^{1}$. Let $S^{[n]}$
be the symmetric algebra of rank $n$ of $\left(V,c\right)$ as in
Definition \ref{def:SymAlg}. Now, any homogeneous primitive element
of degree greater than one in $T\left(V,c\right)$ goes via $\xi_{U}\vartheta_{U}$
in a primitive element of the same degree in $\mathrm{gr}U$. Such
an element is zero as nonzero primitive elements in $\mathrm{gr}U$
are concentrated in degree one (see Remark \ref{rem:strictisprimgen}). Thus $\chi_{U}^{\left[0\right]}=\xi_{U}\pi:S^{\left[0\right]}\rightarrow\mathrm{gr}U$
quotients to a graded braided bialgebra homomorphism $\chi_{U}^{\left[1\right]}:S^{\left[1\right]}\rightarrow\mathrm{gr}U$.
 By the same argument, $\chi_{U}^{\left[1\right]}$ sends to zero
all primitive elements of degree grater then one in $S^{\left[1\right]}$
so that $\chi_{U}^{\left[1\right]}$ quotients to $\chi_{U}^{\left[2\right]}:S^{\left[2\right]}\rightarrow\mathrm{gr}U$
and so on. Finally one gets a graded braided bialgebra homomorphism
$\chi_{U}=\chi_{U}^{\left[\infty\right]}:S^{\left[\infty\right]}=\mathcal{B}\left(V,c\right)\rightarrow\mathrm{gr}U$.
Note that $P\left(\mathcal{B}\left(V,c\right)\right)$ identifies
with $\left(V,c\right)$ via the canonical injection so that $\chi_{B}$
restricted to $P\left(\mathcal{B}\left(V,c\right)\right)$ is the
map $\chi_{U}^{1}$. Since $\pi_{U}\left(v\right)\in P\left(U\right)$, for all $v\in V$,
we have $\mathrm{Ker}(\chi_{U}^{1})=\mathrm{Ker}(\left(\pi_{U})_{\mid V}\right)=0$.
Thus, by Lemma \ref{lem:Heyneman-Radford}, $\chi_{U}$ is injective. \end{proof}
\begin{defn}
\label{def:PBW type} Let $\left(V,c,b\right)$ be a braided bialgebra
and set $U:=U(V,c,b)$. Inspired by \cite[Definition, page 316]{Braverman-Gaitsgory},
we will say that $U$ is of \textbf{Poincaré-Birkhoff-Witt (PBW) type}
whenever the canonical projection $\vartheta_{U}:T\left(V,c\right)\rightarrow\mathfrak{G}\left(U\right)$
of Proposition \ref{pro: induced filtration} quotients to a braided
bialgebra isomorphism $\omega_{U}:\mathcal{B}\left(V,c\right)\rightarrow\mathfrak{G}\left(U\right)$
(compare with \cite[page 92]{Humphreys} for motivating this terminology)
i.e. \eqref{diag:omega} commutes.
\end{defn}
Next aim is to provide a large class of braided vector spaces which
give rise to universal enveloping algebras of PBW type.
\begin{rem}
\label{rem:Universal}Let $\left(V,c,b\right)$ be a braided Lie algebra.
Assume that $\left(V,c\right)$ has combinatorial rank at most one,
see Definition \ref{def:SymAlg}. Let $\beta:E\left(V,c\right)\rightarrow V$
be the restriction of $b^{[0]}$ to the space $E\left(V,c\right)$ spanned
by primitive elements of $T\left(V,c\right)$ of homogeneous degree
at least two. By \cite[Theorem 5.10]{Ardizzoni-MMPrim}, $\left(V,c,\beta\right)$
is a braided Lie algebra in the sense of \cite[Definition 4.1]{Ardizzoni-Universal}
and the corresponding universal enveloping algebra \[
\mathbb{U}\left(V,c,\beta\right):=\frac{T(V,c)}{\left((\mathrm{Id}-\beta)[E\left(V,c\right)]\right)}\]
coincides with $U:=U\left(V,c,b\right)$ (the class $\mathcal{S}$,
appearing in \cite[Theorem 5.10]{Ardizzoni-MMPrim}, is exactly the
class of braided vector spaces of combinatorial rank at most one). \end{rem}
\begin{thm}
\label{thm:comb1PBW}Let $\left(V,c,b\right)$ be a braided Lie algebra.
If $\left(V,c\right)$ has combinatorial rank at most one, then $U\left(V,c,b\right)$
is of PBW type.\end{thm}
\begin{proof}
By \cite[Theorem 5.4]{Ardizzoni-Universal}, $\mathbb{U}\left(V,c,\beta\right)$
is of PBW type in the sense of \cite[Definition 4.14]{Ardizzoni-Universal},
i.e. the projection $\theta:S(V,c)\rightarrow\mathfrak{G}\left(U\right)$,
that makes the diagram\[
\xymatrix@R=10pt{T(V,c)\ar[rr]^{\pi_{S}}\ar[rd]_{\vartheta_{U}} &  & S(V,c)\ar@{.>}[ld]^{\theta}\\
 & \mathfrak{G}(U)}
\]
commutative, is indeed an isomorphism (here $S(V,c)=S^{[1]}$ as in
Definition \ref{def:SymAlg} and $\pi_{S}$ denotes the canonical
projection). Now, since $\left(V,c\right)$ has combinatorial rank
at most one, we have that $S(V,c)=\mathcal{B}\left(V,c\right)$ and
$\pi_{S}=\Omega$. Thus $U$ is PBW type in the sense of Definition
\ref{def:PBW type}.\end{proof}
\begin{example}
\label{ex:ClassicLie}Assume $\mathrm{char}K=0$. Let $L$ be an
ordinary Lie algebra. Consider $L$ as a braided vector space through
the canonical flip map $c:L\otimes L\rightarrow L\otimes L,c(x\otimes y)=y\otimes x$. 

By \cite[Example 6.10]{Ardizzoni-MMPrim}, there exists a bracket
$b$ on $(L,c)$ such that $\left(L,c,b\right)$ is a braided Lie
algebra and $U(L,c,b)$ coincide with the ordinary universal enveloping
algebra $U:=U(L)$. Now, since $\mathrm{char}K=0$, we have that $\left(V,c\right)$
has combinatorial rank at most one (cf. \cite[Theorem 6.13]{Ardizzoni-Sdeg}).
By Theorem \ref{thm:comb1PBW}, $U(L,c,b)$ is of PBW type. Hence
there is braided bialgebra isomorphism $\omega_{U}:\mathcal{B}\left(L,c\right)\rightarrow\mathfrak{G}\left(U\right)$
such that \eqref{diag:omega} commutes. Now, since $\left(L,c\right)$
has combinatorial rank at most one, we have that $\mathcal{B}\left(L,c\right)=S(L)$
(cf. \cite[Remark 6.14]{Ardizzoni-SFFS}). Hence, the fact that $U(L,c,b)$
is of PBW type just means that the classical PBW theorem holds, see
\cite[page 92]{Humphreys}. \end{example}
\begin{example}
\label{exa:RestrictedLie}Assume $\mathrm{char}K$ is a prime number
$p$. Let $V$ be a vector space regarded as a braided vector space
through the canonical flip map $c$. Consider the restricted symmetric
algebra of $V$:\[
\mathfrak{s}(V):=\frac{T(V)}{(xy-yx,x^{p}\mid x,y\in V)}.\]
Now $\mathfrak{s}(V)$ is nothing but the restricted enveloping algebra
of the trivial restricted Lie algebra $V$ so that, by \cite[Theorem 6.11]{Milnor-Moore},
we have that $P(\mathfrak{s}(V))\cong V$. Since $\mathfrak{s}(V)$
is obtained dividing out $T(V,c)$ by elements in $E(V,c)$, by \cite[Theorem 6.1 and Remark 4.3]{Ardizzoni-Sdeg},
$\left(V,c\right)$ has combinatorial rank at most one and $\mathfrak{s}(V)=\mathcal{B}\left(V,c\right)$.

Let $b$ be such that $(V,c,b)$ is a braided Lie algebra. By Remark
\ref{rem:Universal}, \[
U\left(V,c,b\right)=\mathbb{U}\left(V,c,\beta\right)=\frac{T(V,c)}{\left((\mathrm{Id}-\beta)[E\left(V,c\right)]\right)}.\]
Since $xy-yx,x^{p}\in E(V,c)$, for all $x,y\in V$, and the domain
of $\beta$ is $E\left(V,c\right)$, it makes sense to set \begin{eqnarray*}
[x,y] & := & \beta(xy-yx),\qquad x^{[p]}:=\beta(x^{p}).\end{eqnarray*}
This defines two maps $[-,-]:V\otimes V\rightarrow V$ and $-{}^{[p]}:V\rightarrow V$.
It is straightforward to check that $(V,[-,-],-{}^{[p]})$ is a restricted
Lie algebra, see \cite[Definition 4, page 187]{Jacobson-LieAlg}.
{}Consider the restricted enveloping algebra of $V$: \[
\mathfrak{u}(V):=\frac{T(V)}{\left(xy-yx-[x,y],x^{p}-x^{[p]}\mid x,y\in V\right)}.\]
Clearly there exists a projection $\lambda:\mathfrak{u}(V)\rightarrow U\left(V,c,b\right)$.
Since, by \cite[Theorem 6.11]{Milnor-Moore}, we have $P(\mathfrak{u}(V))\cong V$,
we can apply Lemma \ref{lem:Heyneman-Radford} to conclude that $\lambda$
is bijective whence $\mathfrak{u}(V)=U\left(V,c,b\right)$.

Conversely, let $[-,-]:V\otimes V\rightarrow V$ and $-{}^{[p]}:V\rightarrow V$
be such that $(V,[-,-],-{}^{[p]})$ is a restricted Lie algebra and
set $A:=\mathfrak{u}(V)$. This is an ordinary Hopf algebra, see \cite[page 23]{Montgomery}.
In particular it is a braided bialgebra with braiding the canonical
flip map on $A$. It is primitively generated as, by construction
it is generated by the image of $V$ in $A$. By Remark \ref{rem:Milnor-Moore},
$A\cong U\left(P,c_{P},b_{P}\right)$ where $\left(P,c_{P},b_{P}\right)$
is the infinitesimal braided Lie algebra of $A$. 

Now the canonical map $\sigma:V\rightarrow P=P\left(A\right)$ is
bijective (cf. \cite[Theorem 6.11]{Milnor-Moore}). Since $c_{P}$
is the restriction of the braiding of $A$, then $c_{P}$ is the canonical
flip map on $P$. Thus $\sigma:(V,c)\rightarrow(P,c_{P})$ is an isomorphism
of braided vector spaces, where $c$ is the canonical flip map on
$V$. Hence we can endow $(V,c)$ with a bracket $b$ such that $(V,c,b)$
is a braided Lie algebra and $\sigma:(V,c,b)\rightarrow(P,c_{P},b_{P})$
is an isomorphism of braided Lie algebras. Hence $A\cong U:=U\left(V,c,b\right)$.
By the foregoing, $(V,c)$ has combinatorial rank at most one so that,
by Theorem \ref{thm:comb1PBW}, $U$ is of PBW type. Hence there is
a braided bialgebra isomorphism $\omega_{U}:\mathcal{B}\left(V,c\right)\rightarrow\mathfrak{G}\left(U\right)$
such that \eqref{diag:omega} commutes. Since, by the initial part,
one has that $\mathcal{B}\left(V,c\right)=\mathfrak{s}(V)$, then
we get an isomorphism $\omega_{U}:\mathfrak{s}(V)\rightarrow\mathfrak{G}\left(\mathfrak{u}(V)\right)$
which is a PBW theorem for restricted enveloping algebras, see \cite[Proposition 6.12]{Milnor-Moore}. 
\end{example}
\section{Basis for the universal enveloping algebra \label{sec:PBWbasis}}

In this section, we will discuss the problem of determining a basis
for a universal enveloping algebra of PBW type.
\begin{lem}
\label{lem:Dragos} Let $A,$ $B,$ $C$ be vector spaces, and let
$W\leq A$ and $B^{\prime}\leq B$ be two vector subspaces. Let $f,$
$g,$ $p$ and $q$ be $K$-linear maps as in the following commutative
diagram: \[
\xymatrix{A\ar[d]_{p}\ar[r]^{f} & B\ar[d]^{q}\\
C\ar[r]^{g} & \frac{B}{B^{\prime}}}
\]
 where $q$ is the canonical projection. If $g$ and $p|_{W}:W\rightarrow C$
are isomorphisms then $f(W)\oplus B^{\prime}=B$. Moreover $f_{\mid W}$
is injective.\end{lem}
\begin{proof}
The first part of the statement is exactly \cite[Lemma 4.21]{AMS-MM}.
Let $w\in W\cap\mathrm{Ker}f$. Then $gp\left(w\right)=qf\left(w\right)=0$.
Since $g$ is bijective, we get $p\left(w\right)=0$. Since $p|_{W}:W\rightarrow C$
is an isomorphism, we conclude that $w=0$.\end{proof}
\begin{prop}
\label{pro:PBWbasis}Let $\left(V,c,b\right)$ be a braided Lie algebra
and set $U:=U(V,c,b)$. Let $\Omega:T\left(V,c\right)\rightarrow\mathcal{B}\left(V,c\right)$
be the canonical projection and denote by $\Omega^{n}:V^{\otimes n}\rightarrow\mathcal{B}\left(V,c\right)^{n}$
the $n$th graded component of $\Omega$. Assume that $U$ is of PBW
type and let $W_{n}$ be a vector subspace of $V^{\otimes n}$ such
that $\Omega_{\mid W_{n}}^{n}$ is an isomorphism. Then $U_{\left(n\right)}=U_{\left(n-1\right)}\oplus\pi_{U}\left(W_{n}\right)$, where $\left(U_{\left(n\right)}\right)_{n\in\mathbb{N}}$ is the
standard filtration on $U$. Moreover $\pi_{U\mid W_{n}}$
is injective.\end{prop}
\begin{proof}
Clearly $\pi_{U}\left(V^{\otimes n}\right)\subseteq U_{\left(n\right)}$
so $\pi_{U}$ induces a map $\tilde{\pi}^{n}:V^{\otimes n}\rightarrow U_{\left(n\right)}$.
Apply Lemma \ref{lem:Dragos} to the following diagram \[
\xymatrix{V^{\otimes n}\ar[d]_{\Omega^{n}}\ar[r]^{\tilde{\pi}^{n}} & U_{\left(n\right)}\ar[d]^{q^{n}}\\
\mathcal{B}\left(V,c\right)^{n}\ar[r]^{\omega_{U}^{n}} & \frac{U_{\left(n\right)}}{U_{\left(n-1\right)}}}
\]
where $q^{n}:U_{\left(n\right)}\rightarrow U_{\left(n\right)}/U_{\left(n-1\right)}$
is the canonical projection and $\omega_{U}^{n}$ is the $n$th graded
component of the isomorphism $\omega_{U}:\mathcal{B}\left(V,c\right)\rightarrow\mathfrak{G}\left(U\right)$
of Definition \ref{def:PBW type}.\end{proof}
\begin{rem}
\label{rem:PBWbasis}Let $\left(V,c,b\right)$ be a braided Lie algebra
and set $U:=U\left(V,c,b\right)$. Assume that $U$ is of PBW type.
Set $\Theta:=\{n\in\mathbb{N}\mid\mathcal{B}\left(V,c\right)^{n}\neq\{0\}\}$
and suppose that, for each $n\in\Theta$, we can find a linearly independent
set $Z_{n}:=\left\{ v_{n,i}\mid i\in I_{n}\right\} $ consisting of
elements of $V^{\otimes n}$ with the property that $W_{n}:=\mathrm{Span}{}_{K}Z_{n}$
is such that $\Omega_{\mid W_{n}}^{n}$ is an isomorphism. By Proposition
\ref{pro:PBWbasis}, $U_{\left(n\right)}=U_{\left(n-1\right)}\oplus\pi_{U}(W_{n})$.
Moreover $\pi_{U\mid W_{n}}$ is injective. Therefore $\left\{ \pi_{U}\left(v_{n,i}\right)\mid n\in\Theta,i\in I_{n}\right\} $
is a basis for $U$. Hence, as in the classical case, finding a basis
for $U$ reduces to determine a suitable basis $\left\{ v_{n,i}\mid n\in\Theta,i\in I_{n}\right\} $
for $\mathcal{B}\left(V,c\right)$ as above. Results in this direction
are obtained by Kharchenko (cf. \cite[Theorem 2]{Kherchenko-Aquantum})
and more generally by Ufer (cf. \cite[Theorem 34]{Ufer}) when $(V,c)$
is a braided vector space of diagonal type or left triangular respectively.
\end{rem}
Next aim is show how the computation of a PBW basis for the classical
universal enveloping algebra and for the restricted enveloping algebra
fits into the theory above.
\begin{example}
\label{exa:PBWbasis-char0}Assume $\mathrm{char}K=0$. Let $L$ be
an ordinary Lie algebra which is assumed to have a totally ordered
basis $(X,\leq)$. By Example \ref{ex:ClassicLie}, the universal
enveloping algebra $U:=U(L)$ can be identified with the the universal
enveloping algebra $U(L,c,b)$ where $c$ is the canonical flip map
on $L$ and $b$ a suitable bracket for $(L,c)$. Moreover $(L,c)$
has combinatorial rank at most one and $\mathcal{B}\left(L,c\right)$
is the ordinary symmetric algebra $S(L)$. Consider the canonical
projection $\Omega:T\left(L\right)\rightarrow S\left(L\right)$. For
$n=0$, set $Z_{0}:=\{1_{K}\}$ and, for $n\geq1$, let $Z_{n}\subseteq L^{\otimes n}$
be the set \[
\left\{ x_{1}\cdot_{T}\cdots\cdot_{T}x_{n}\mid x_{i}\in X,\forall1\leq i\leq n\text{ and }x_{1}\leq x_{2}\leq\cdots\leq x_{n}\right\} .\]
If we set $W_{n}:=\mathrm{Span}{}_{K}Z_{n}$, it is clear that $\Omega_{\mid W_{n}}^{n}$
is an isomorphism. Hence, by Remark \ref{rem:PBWbasis}, The elements
$x_{1}\cdot_{U}\cdots\cdot_{U}x_{n}$, where $n\geq1$, $x_{i}\in X$,
for all $1\leq i\leq n$, and $x_{1}\leq x_{2}\leq\cdots\leq x_{n}$,
along with $1_K$, form a basis of the universal enveloping algebra
$U(L)$ of $L$. This theorem is due to Poincaré, Birkhoff and Witt
and the basis is called the PBW basis of the universal enveloping
algebra (see e.g. \cite[Corollary C, page 92]{Humphreys}). \end{example}
\begin{rem}
Keep the assumptions and notations of Example \ref{exa:PBWbasis-char0}.
Observe that proving the isomorphism $U(L)\cong U(L,c,b)$ requires
the condition $PU(L)\cong L$ (cf. \cite[Example 6.10]{Ardizzoni-MMPrim}).
One could object that, in order to check this isomorphism, a basis
of $U(L)$ is needed. We can clear the hurdle as follows. By \cite[page 92]{Humphreys},
$U(L)$ fulfills the PBW theorem. Now, mimicking the proof of \cite[Corollary 5.5]{Ardizzoni-Universal},
we arrive at $PU(L)\cong L$ without using a basis of $U(L)$.\end{rem}
\begin{example}
\label{exa:PBWbasis-charp}Assume $\mathrm{char}K$ is a prime number
$p$. Let $(L,[-,-],-{}^{[p]})$ be a restricted Lie algebra which
is assumed to have a totally ordered basis $(X,\leq)$. By Example
\ref{exa:RestrictedLie}, the restricted enveloping algebra $\mathfrak{u}:=\mathfrak{u}(L)$
can be identified with the the universal enveloping algebra $U(L,c,b)$
where $c$ is the canonical flip map on $L$ and $b$ a suitable bracket
for $(L,c)$. Moreover $(L,c)$ has combinatorial rank at most one.
Hence by Theorem \ref{thm:comb1PBW}, $U(L,c,b)$ is of PBW type.
Furthermore, by Example \ref{exa:RestrictedLie}, $\mathcal{B}\left(L,c\right)$
is the restricted symmetric algebra $\mathfrak{s}(L)$. Consider the
canonical projection $\Omega:T\left(L\right)\rightarrow\mathfrak{s}(L)$.
For $n=0$, set $Z_{0}:=\{1_{K}\}$ and, for $n\geq1$, let $Z_{n}\subseteq L^{\otimes n}$
be the set \[
\left\{ x_{1}^{t_{1}}\cdot_{T}\cdots\cdot_{T}x_{n}^{t_{n}}\mid x_{i}\in X,0\leq t_{i}\leq p-1,\forall1\leq i\leq n,t_{1}+\cdots+t_{n}=n\text{ and }x_{1}<x_{2}<\cdots<x_{n}\right\} .\]
Then $Z_{n}=\emptyset$ whenever $\mathcal{B}\left(L,c\right)^{n}=\{0\}$
and $Z_{n}$ is linearly independent otherwise. If we set $W_{n}:=\mathrm{Span}{}_{K}Z_{n}$,
it is clear that $\Omega_{\mid W_{n}}^{n}$ is an isomorphism for
all $n\in\mathbb{N}$. Hence, by Remark \ref{rem:PBWbasis}, The elements
$x_{1}^{t_{1}}\cdot_{\mathfrak{u}}\cdots\cdot_{\mathfrak{u}}x_{n}^{t_{n}}$,
where $n\geq1$, $x_{i}\in X,0\leq t_{i}\leq p-1$, for all $1\leq i\leq n$,
and $x_{1}<x_{2}<\cdots<x_{n}$, along with $1_K$, form a basis of
$\mathfrak{u}(L)$. This basis is called the PBW basis of the restricted
enveloping algebra (see e.g. \cite[page 23]{Montgomery}). 
\end{example}
We now give an example of a PBW basis for the universal enveloping
algebra of a braided Lie algebra whose bracket $c$ is not a symmetry
i.e. $c^{2}\neq\mathrm{Id}$.
\begin{example}
\label{exa:PBWbasis-Stumbo}Assume $\mathrm{char}K=0$. Let $V:=Kx_{1}\oplus Kx_{2}$
and define a diagonal braiding $c$ on $V$ by setting $c(x_{i}\otimes x_{j})=q_{ij}x_{j}\otimes x_{i}$,
where $q_{11}=\gamma\in K$ and $q_{ij}=1$ for all $(i,j)\neq(1,1)$.
Assume $\gamma$ is not a root of unity. In view of \cite[Example 9.8]{Ardizzoni-Universal},
the endomorphism $c$ has minimal polynomial $(X-\gamma)(X^{2}-1)$
(whence it is not a symmetry), the braided vector space $(V,c)$ has
combinatorial rank at most one and Nichols algebra $\mathcal{B}\left(V,c\right)=T(V,c)/(x_{2}x_{1}-x_{1}x_{2})$.
Consider the braided bialgebra $A:=T(V,c)/(x_{2}x_{1}-x_{1}x_{2}-x_{1})$.
Still by \cite[Example 9.8]{Ardizzoni-Universal}, we know that the
infinitesimal part of $A$ identifies with $(V,c)$. Thus, by Remark
\ref{rem:Milnor-Moore}, there is a braided bracket $b$ on $(V,c)$
such that $(V,c,b)$ is a braided Lie algebra and $A\cong U(V,c,b)$.
By Theorem \ref{thm:comb1PBW}, $U(V,c,b)$ is of PBW type.

For all $n\in\mathbb{N}$, set $Z_{n}:=\left\{ x_{1}^{t_{1}}\cdot_{T}x_{2}^{t_{2}}\mid t_{1},t_{2}\in\mathbb{N},t_{1}+t_{2}=n\right\} \subseteq V^{\otimes n}$
. If we set $W_{n}:=\mathrm{Span}{}_{K}Z_{n}$, it is clear that $\Omega_{\mid W_{n}}^{n}$
is an isomorphism for all $n\in\mathbb{N}$. Hence, by Remark \ref{rem:PBWbasis},
The elements $x_{1}^{t_{1}}\cdot_{A}x_{2}^{t_{2}}$, where $t_{1},t_{2}\in\mathbb{N}$,
form a basis of $A$. It is remarkable that $A$ is not a classical
universal enveloping algebra as its infinitesimal braiding is not
a symmetry. Note also that, as we will see in Theorem
\ref{teo:PBW}, since $U(V,c,b)$ is of PBW type then $U$ is strictly generated by $V$. Hence
the $n$th term of coradical filtration of $A$ is $\sum_{i=0}^{n}\pi_{A}\left(V\right)^{i}$,
where $\pi_{A}:T(V,c)\rightarrow A$ denotes the canonical projection.
Moreover, by Theorem \ref{thm:Lifting}, $A$ is what we will call a braided bialgebra
lifting of $\mathcal{B}\left(V,c\right)$.\end{example}

\section{Cosymmetric \label{sec:Cosymm}}

We will investigate universal enveloping algebras which are cosymmetric
in the sense of \cite[Definition 3.1]{Kharchenko- connected}.
\begin{defn}
Let $\left(C,\Delta{}_{C},\varepsilon{}_{C},c\right)$ be a connected
braided coalgebra. In view of \cite[Remark 1.12]{AMS-MM2}, the braiding
of $C$ induces a braiding $c_{P}$ on the space $P:=P\left(C\right)$
of primitive elements in $C$. Moreover $\left(P,c_{P}\right)$ is
a braided vector space. Let $1_{C}$ be the unique group-like element
of $C.$ Let $\phi:C\rightarrow C$ be defined by $\phi\left(c\right):=c-\varepsilon_{C}\left(c\right)1_{C}$
and let $\Delta_{C}^{n}:C\rightarrow C^{\otimes (n+1)}$ be the $n$th
iterated comultiplication of $C$. For $n\in \mathbb{N}$, let $C_{n}$ be the $n$th term
of the coradical filtration of $C$. By \cite[Lemma 2.2]{Kharchenko- connected}, which is deduced from \cite[Proposition 11.0.5]{Sw},
we have $\mathrm{Ker}(\phi^{\otimes (n+1)}\Delta_{C}^{n})=C_{n}$,
for all $n\in\mathbb{N}$. Moreover $\phi^{\otimes (n+1)}\Delta{}_{C}^{n}\left(C_{n+1}\right)\subseteq P^{\otimes (n+1)}$
so that, for all $n>0$, the restriction of $\phi^{\otimes n}\Delta{}_{C}^{n-1}$
quotients to a map $\mu_{n}:C_{n}/C_{n-1}\rightarrow P^{\otimes n}$.
Set $\mu_{0}:=\mathrm{Id}_{K}$. Then $\mu_{C}:=\oplus_{n\in\mathbb{N}}\mu_{n}:\mathrm{gr}C\rightarrow T^{c}\left(P,c_{P}\right)$
is called the \textbf{linearization map} and it is an injective coalgebra
map. Here $T^{c}\left(P,c_{P}\right)$ denotes the braided cotensor
coalgebra of $(P,c_{P})$ also known as quantum shuffle algebra and
denoted by $\mathrm{Sh}_{c_{P}}(P)$. The linearization map can equivalently
be constructed by means of the universal property of the cotensor
coalgebra $T^{c}\left(P\right)$ \cite[Proposition 12.1.1]{Sw}. Following
\cite[Definition 3.1]{Kharchenko- connected} we will say that a connected
braided bialgebra $C$ is \textbf{cosymmetric} if $\mathrm{Im}\left(\mu_{C}\right)\subseteq\mathcal{B}\left(V,c\right)$
(note that in \cite{Kharchenko- connected}, the Nichols algebra $\mathcal{B}\left(V,c\right)$
is denoted by $S_{c}(V)$). %
{}

For a connected braided bialgebra $B$, one has that the linearization
map $\mu_{B}:\mathrm{gr}B\rightarrow T^{c}\left(P,c_{P}\right)$ is
indeed a braided bialgebra map (cf. \cite[Proposition 3.3]{Kharchenko- connected}).\end{defn}
\begin{thm}
\label{teo:PBW}Let $\left(V,c,b\right)$ be a braided Lie algebra
and set $U:=U(V,c,b)$. The following assertions are equivalent.
\begin{enumerate}
\item $U$ is of PBW type.
\item $U$ is cosymmetric.
\item $U$ is strictly generated by $V$.
\item $\mathrm{gr}U$ is primitively generated.
\item The map $\chi_{U}:\mathcal{B}\left(V,c\right)\rightarrow\mathrm{gr}U$
of Proposition \ref{pro:omegaB} is bijective.
\end{enumerate}
\end{thm}
\begin{proof}
By Remark \ref{rem:ConnIsHopf}, the primitively generated braided
bialgebra $U$ is indeed a braided Hopf algebra. Hence, the equivalence
between $\left(2\right),\left(3\right),\left(4\right)$ and $\left(5\right)$
follows by \cite[Theorem 3.5]{Kharchenko- connected}. 

Let $\left(U_{\left(n\right)}\right)_{n\in\mathbb{N}}$ be the standard
filtration on $U$. Let $\vartheta_{U}:T\left(V,c\right)\rightarrow\mathfrak{G}\left(U\right)$
be the graded braided bialgebra homomorphism of Proposition \ref{pro: induced filtration}.

$\left(1\right)\Rightarrow\left(3\right)$ By hypothesis, $\vartheta_{U}$
quotients to a braided bialgebra isomorphism $\omega_{U}:\mathcal{B}\left(V,c\right)\rightarrow\mathfrak{G}\left(U\right)$.
We have $P\left(\mathfrak{G}\left(U\right)\right)=\omega_{U}\left(P\left(\mathcal{B}\left(V,c\right)\right)\right)=\omega_{U}\left(V\right)=U_{\left(1\right)}/U_{\left(0\right)}$.
By Proposition \ref{pro:CoradicalFiltration}, we get that $U$ is
strictly generated by $V$.

$\left(5\right)\Rightarrow\left(1\right)$ By Proposition \ref{pro:omegaB},
we have $\chi_{U}\Omega=\xi_{U}\vartheta_{U}$. Since both $\chi_{U}$
and $\Omega$ are surjective, so is $\xi_{U}$. By Proposition \ref{pro:CoradicalFiltration},
$\xi_{U}$ is bijective. Set $\omega_{U}:=\left(\xi_{U}\right)^{-1}\chi_{U}$.\end{proof}
\begin{rem}
In Theorem \ref{thm:comb1PBW}, we gave a class of braided vector
spaces such that $U\left(V,c,b\right)$ is of PBW type whatever is
the bracket. It is still an open question whether the property of
$U\left(V,c,b\right)$ to be of PBW type depends just on the braided
vector space $\left(V,c\right)$ in general. In Example \ref{ex:Cosymm}
we will exhibit a braided Lie algebra $\left(V,c,b\right)$ such that
$U\left(V,c,b\right)$ is not of PBW type.\end{rem}

\begin{thm}
\label{thm:KhaCosym}Let $\left(V,c,b\right)$ be a braided Lie algebra
and set $U:=U\left(V,c,b\right)$. Assume $U^{\left[n\right]}$ is
cosymmetric for some $n\in\mathbb{N}$. Then $U^{\left[n+1\right]}$
is cosymmetric and $U\left(V,c,b\right)=U^{\left[n+1\right]}$. \end{thm}

\begin{proof}
By construction (cf. the proof of \cite[Proposition 4.5]{Ardizzoni-MMPrim}), we have that, for all $n\in\mathbb{N}$, $K^{[n]}:=\mathrm{Im}\left(\mathrm{Id}_{P^{\left[n\right]}}-i^{\left[n\right]}b^{\left[n\right]}\right)$ is a categorical subspace of $P^{\left[n\right]}$. For all $n\in\mathbb{N}$, set $W^{\left[n\right]}:=\mathrm{Ker}\left(\pi_{n}^{n+1}\right)\cap P^{\left[n\right]}$

By Lemma \ref{lem:Psplits}, $W^{\left[n\right]}=K^{[n]}$ and $P^{\left[n\right]}=W^{\left[n\right]}\oplus V^{\left[n\right]}$, for all  $n\in\mathbb{N}$. Assume now that $U^{\left[n\right]}$ is
cosymmetric for a fixed $n\in\mathbb{N}$.  By \cite[Lemmata 4.2 and 4.4]{Kharchenko- connected}, taking $W^{\left[n\right]}$ as $W$ and $V^{\left[n\right]}$ as $W'$,
we get that $U^{\left[n+1\right]}$ is cosymmetric and $P^{\left[n+1\right]}=\pi_{n}^{n+1}(P^{\left[n\right]})$.
Now\[
P^{\left[n+1\right]}=\pi_{n}^{n+1}(P^{\left[n\right]}) =\pi_{n}^{n+1}i^{\left[n\right]}b^{\left[n\right]}(P^{\left[n\right]}) =i^{\left[n+1\right]}b^{\left[n\right]}(P^{\left[n\right]}) =\mathrm{Im}(i^{\left[n+1\right]})=V^{\left[n+1\right]}.\]
Hence $W^{\left[n+1\right]}=0$. Now, by definition of $\pi_{n+1}^{n+2}$, we have  $\mathrm{Ker}\left(\pi_{n+1}^{n+2}\right)=\left(K^{[n+1]}\right)$ so that, by the foregoing, we get $\mathrm{Ker}\left(\pi_{n+1}^{n+2}\right)=\left(W^{\left[n+1\right]}\right)=0$. Therefore $\pi_{n+1}^{n+2}$ is bijective and $U^{\left[n+1\right]}=U^{\left[n+2\right]}$.
Thus $U\left(V,c,b\right)=\underrightarrow{\lim}U^{\left[i\right]}=U^{\left[n+1\right]}$.\end{proof}

\begin{cor}
\label{cor:KhaCosym}Let $\left(V,c\right)$ be a braided vector space.
If there is $n\in\mathbb{N}$ such that the symmetric algebra $S^{\left[n\right]}$
of rank $n$ is cosymmetric, then $\left(V,c\right)$ has combinatorial
rank at most $n+1$ in the sense of Definition \ref{def:SymAlg}.\end{cor}
\begin{proof}
Set $U:=U\left(V,c,b_{tr}\right)$ where $b_{tr}$ is the trivial
bracket on $(V,c)$. Note that $U^{\left[t\right]}=S^{[t]}$ is the
symmetric algebra of rank $t$ for all $t\in\mathbb{N}$. Hence, by
Theorem \ref{thm:KhaCosym}, we have that $U^{\left[n+1\right]}$
is cosymmetric and $\mathcal{B}\left(V,c\right)=U\left(V,c,b_{tr}\right)=U^{\left[n+1\right]}$.
This means $\left(V,c\right)$ has combinatorial at most $n+1$.\end{proof}
\begin{example}
\label{ex:Cosymm}At the end of \cite{Kharchenko-SkewPrim}, an example
of a two-dimensional braided vector space $\left(V,c\right)$ of combinatorial
rank $2$ is given. The braiding $c$ is of diagonal type of the form
$c\left(x_{i}\otimes x_{j}\right)=q_{i,j}x_{j}\otimes x_{i}$, $1\leq i,j\leq2$,
where $x_{1},x_{2}$ forms a basis of $V$ over $K$, $q_{1,2}=1\neq-1$
and $q_{i,j}=-1$ for all $\left(i,j\right)\neq\left(1,2\right)$.
By Corollary \ref{cor:KhaCosym}, applied to the case $n=0$, $T:=T\left(V,c\right)$
is not cosymmetric.

Since $T$ is a primitively generated braided bialgebra, by Remark
\ref{rem:Milnor-Moore}, we know that $T$ coincides with the universal
enveloping algebra of its infinitesimal braided Lie algebra $\left(P,c_{P},b_{P}\right)$.
We have so exhibited a braided Lie algebra $\left(P,c_{P},b_{P}\right)$
such that $U\left(P,c_{P},b_{P}\right)$ is not cosymmetric. By Theorem
\ref{teo:PBW}, $U\left(P,c_{P},b_{P}\right)$ is not of PBW type.
\end{example}

\section{Lifting of Nichols algebras \label{sec:Lifting}}

In this section we investigate braided bialgebra liftings of Nichols
algebras. We will characterize them in terms of universal enveloping
algebras of PBW type.
\begin{defn}
Let $\left(V,c\right)$ be a braided vector space. We will say that
a braided bialgebra $B$ \textbf{is a lifting of $\mathcal{B}\left(V,c\right)$}
if there is a graded braided bialgebra isomorphism $\chi_{B}:\mathcal{B}\left(V,c\right)\rightarrow\mathrm{gr}B$. \end{defn}
\begin{thm}
\label{thm:Lifting}Let $\left(V,c\right)$ be a braided vector space
and let $B$ be a braided bialgebra. The following assertions are
equivalent.
\begin{enumerate}
\item $B$ is a lifting of $\mathcal{B}\left(V,c\right)$.
\item There is a bracket $b$ on $(V,c)$ such that $(V,c,b)$ is a braided
Lie algebra, $U(V,c,b)$ is of PBW type and $B\cong U(V,c,b)$ as
braided bialgebras.
\end{enumerate}
\end{thm}
\begin{proof}
$(1)\Rightarrow\left(2\right)$ By hypothesis, there is a braided
bialgebra isomorphism $\chi_{B}:\mathcal{B}\left(V,c\right)\rightarrow\mathrm{gr}B$.
Since $\mathcal{B}\left(V,c\right)$ is primitively generated, the
same holds for $\mathrm{gr}B$ whence for $B$. By Remark \ref{rem:Milnor-Moore},
$B\cong U\left(P,c_{P},b_{P}\right)$, where $\left(P,c_{P},b_{P}\right)$
is the infinitesimal braided Lie algebra of $B$. Clearly $\chi_{B}$
preserves primitive elements so that it induces an isomorphism of
braided vector spaces $\gamma:\left(V,c\right)\rightarrow\left(P,c_{P}\right)$.
Since $\gamma$ is bijective, there is a bracket $b$ on $\left(V,c\right)$
such that $\left(V,c,b\right)$ is a braided Lie algebra and $U\left(V,c,b\right)\cong U\left(P,c_{P},b_{P}\right)\cong B$.
Set $U:=U\left(V,c,b\right)$ and denote by $\varphi:U\rightarrow B$
this isomorphism. Then $\varphi_{\mid V}=\gamma$. Thus one easily
checks that the diagram\[
\xymatrix@R=10pt{\mathrm{gr}U\ar[rr]^{\mathrm{gr}\varphi} &  & \mathrm{gr}B\\
 & \mathcal{B}(V,c)\ar[ul]^{\chi_{U}}\ar[ur]_{\chi_{B}}}
\]
commutes. Since both $\chi_{B}$ and $\varphi$ are isomorphisms,
we get that $\chi_{U}$ is an isomorphism too. By Theorem \ref{teo:PBW},
we get that $U$ is of PBW type.

$(2)\Rightarrow\left(1\right)$ Since $U:=U\left(V,c,b\right)$ is
of PBW type, by Theorem \ref{teo:PBW} we have that the map $\chi_{U}:\mathcal{B}\left(V,c\right)\rightarrow\mathrm{gr}U$
of Proposition \ref{pro:omegaB} is bijective. Hence $U$ is a lifting
of $\mathcal{B}\left(V,c\right)$.
\end{proof}
In view of Theorem \ref{thm:Lifting}, given a braided vector space
$\left(V,c\right)$, studying braided bialgebra liftings of $\mathcal{B}\left(V,c\right)$
amounts to investigate braided brackets $b$ on $\left(V,c\right)$
such that $\left(V,c,b\right)$ is a braided Lie algebra and $U\left(V,c,b\right)$
is of PBW type. 
\begin{cor}
\label{cor:Lifting}Let $\left(V,c\right)$ be a braided vector space
of combinatorial rank at most one and let $B$ be a braided bialgebra.
The following assertions are equivalent.
\begin{enumerate}
\item $B$ is a lifting of $\mathcal{B}\left(V,c\right)$.
\item There is a bracket $b$ on $(V,c)$ such that $(V,c,b)$ is a braided
Lie algebra and $B\cong U(V,c,b)$ as braided bialgebras.
\end{enumerate}
\end{cor}
\begin{proof}
It follows by Theorem \ref{thm:Lifting} and Theorem \ref{thm:comb1PBW}.
\end{proof}

\noindent \textbf{Acknowledgements.} We would like to thank the referee for
several useful suggestions that improved an earlier version of this paper.

\end{document}